\documentclass[10pt]{article}
\usepackage{color}
\usepackage{amssymb}
\usepackage{amsthm,array,amssymb,amscd,amsfonts,latexsym, url}
\usepackage{amsmath}
\usepackage{verbatim}
\usepackage{tikz-cd}
\usepackage{enumitem}

\usepackage{MnSymbol}

\usepackage{hyperref}

\usepackage{longtable} 
\usepackage[center]{caption}
\usepackage{diagbox}

\usepackage[all]{xy}

\makeatletter
\def\th@remark{%
  \thm@headfont{\bfseries}%
  \normalfont 
  \thm@preskip\topsep \divide\thm@preskip\tw@
  \thm@postskip\thm@preskip
}
\makeatother

\newtheorem{theo}{Theorem}[section]
\newtheorem{thm}{Theorem}[section]
\newtheorem{prop}[theo]{Proposition}

\newtheorem{lemma}[theo]{Lemma}

\newtheorem{conj}[theo]{Conjecture}

\theoremstyle{remark}

\newtheorem{rmk}[theo]{Remark}
\newtheorem{example}[theo]{Example}

\newcommand{\BC}{{\mathbb{C}}}

\newcommand{\BE}{{\mathbb{E}}}

\newcommand{\BQ}{{\mathbb{Q}}}
\newcommand{\BR}{{\mathbb{R}}}

\newcommand{\BZ}{{\mathbb{Z}}}

\newcommand{\CC}{{\mathcal C}}

\newcommand{\CO}{{\mathcal O}}

\newcommand{\pt}{{\mathsf{p}}}

\newcommand{\blangle}{\big\langle}
\newcommand{\brangle}{\big\rangle}

\newcommand{\Mbar}{{\overline M}}
\newcommand\ev{\operatorname{ev}}

\newcommand{\GW}{\mathsf{GW}}

\DeclareFontFamily{OT1}{rsfs}{}
\DeclareFontShape{OT1}{rsfs}{n}{it}{<-> rsfs10}{}
\DeclareMathAlphabet{\curly}{OT1}{rsfs}{n}{it}

\newcommand{\p}{\mathbb{P}}

\newcommand\End{\operatorname{End}}

\newcommand\wt{\operatorname{wt}}

\newcommand{\vir}{\mathsf{vir}}
\newcommand{\red}{\mathsf{red}}

\newcommand{\QMod}{\mathsf{QMod}}
\newcommand{\Mod}{\mathsf{Mod}}

\title{On the descendent Gromov--Witten theory of a K3 surface}
\author{Georg Oberdieck}

\date{}

\newfont{\gothic}{eufb10}
\begin{document}
\maketitle

\begin{abstract} We study the reduced descendent Gromov--Witten theory of K3 surfaces in primitive curve classes. We present a conjectural closed formula for the stationary theory, which generalizes the Bryan-Leung formula. We also prove a new recursion that allows to remove descendent insertions of $1$ in many instances.
Together this yields an efficient way to compute a large class of invariants (modulo the conjecture on the stationary part).
As a corollary we conjecture a surprising polynomial structure which underlies the Gromov--Witten invariants of the K3 surface.
 \end{abstract}

\setcounter{tocdepth}{1} 


\setcounter{section}{-1}

\section{Introduction}
\subsection{State of the art}
Let $S$ be a K3 surface and let $\beta \in H_2(S,\BZ)$ be an effective curve class.
The (reduced) descendent Gromov--Witten invariants of $S$ are defined by integrating over the
moduli space of $n$-marked genus $g$ degree $\beta$ stable maps:
\begin{equation} \label{GWinvariant}
\left\langle \tau_{k_1}(\gamma_1) \cdots \tau_{k_n}(\gamma_n) \right\rangle^{S}_{g,\beta}
=
\int_{[ \Mbar_{g,n}^{\circ}(S,\beta) ]^{\red}} \prod_i \ev_i^{\ast}(\gamma_i) \psi_i^{k_i},
\end{equation}
where $k_1, \ldots, k_n \geq 0$ and $\gamma_1, \ldots, \gamma_n \in H^{\ast}(S)$.
We refer to Section~\ref{subsec:GW theory of K3 definitions} for more details on the definition.
We say that the descendent invariant \eqref{GWinvariant} is:
\begin{itemize}[itemsep=0pt]
\item {\em stationary} if $\deg(\gamma_i) > 0$ for all $i$,
\item {\em primitive} if the curve class $\beta \in H_2(S,\BZ)$ is primitive.
\end{itemize}

For dimension reasons 
the invariant \eqref{GWinvariant} vanishes unless 
we have\footnote{Here $\deg_{\BC}(\gamma)$ is the complex cohomological degree of $\gamma$, that is $\gamma \in H^{2 \deg_{\BC}(\gamma)}(S)$.}
\[ g = \sum_{i=1}^{n} (k_i + \deg_{\BC}(\gamma_i) - 1). \]
Hence we fix $g$ by this constraint and often drop it from notation.

The most important conjecture about the Gromov--Witten theory of the K3 surface
says that the descendent invariants are completely determined by the primitive invariants.
We recall the conjecture.
Let $\pt \in H^4(S,\BZ)$ denote the class of a point.

\begin{conj}[{Multiple cover conjecture \cite[Conj. C2]{K3xE}}]
For every positive divisor $k|\beta$ let $S_k$ be a K3 surface and let $\varphi_k : H^2(S,\BR) \to H^2(S_k,\BR)$ be a real isometry such that $\varphi_k(\beta/k)$ is a primitive effective curve class. 
Extend $\varphi_k$ to an isomorphism $\varphi_k : H^{\ast}(S,\BR) \to H^{\ast}(S_k,\BR)$ by setting $\varphi_k(1) = 1$ and $\varphi_k(\pt) = \pt$.
Then we have:
\[
\left\langle \tau_{k_1}(\gamma_1) \cdots \tau_{k_n}(\gamma_n) \right\rangle^{S}_{g,\beta}
=
\sum_{k|\beta}
k^{2g-3+\sum_{i=1}^{n} \deg_{\BC}(\gamma_i)}
\left\langle \tau_{k_1}(\varphi_k(\gamma_1)) \cdots \tau_{k_n}(\varphi_k(\gamma_n)) \right\rangle^{S_k}_{g,\varphi_k(\beta/k)}.
\]
\end{conj}

The conjecture was proven by Bae and B\"ulles \cite{BB} when $\beta$ has divisibility $2$,
but remains wide open for higher divisibility.

Although the imprimitive invariants are difficult to understand, the situation for the primitive invariants is much better.
Indeed, Maulik, Pandharipande and Thomas provided in \cite{MPT} an algorithm which can determine all primitive invariants using a combination of tautological relations coming from the moduli space of curves and degenerations techniques. 
As a corollary they showed that 
the natural
 generating series of primitive invariants
are quasi-modular forms.
Further in \cite{HAE} it was shown using this algorithm that these quasi-modular forms satisfy a holomorphic anomaly equation. We review these results in Section~\ref{sec:Background}.

One could say that the story is finished here and the primitive invariants are completely determined. 
However, there are two problems:
First, the algorithm of \cite{MPT} is extremely complicated and increasingly slow when the genus grows.
We refer to work of Sendra \cite{Sendra} where this algorithm was implemented. Computations in his implementation are feasible only up to genus $g=3$.
Second, very few explicit formulas for the primitive invariants are known.
This is in strong contrast to the case of elliptic curves, where the Bloch-Okounkov formula explicitly evaluates all invariants 
in closed form \cite{OP, OP3, Pixton}.
Therefore, when it comes to actual computations,
the structure of the primitive invariants of the K3 surface is still very mysterious.
At this point, the only general formula for the descendent invariants
is the following beautiful
result of Bryan and Leung:\footnote{We restrict ourselves here to the pure descendent invariants.
There are more formulas known if one allows more general insertions, such as the Hodge classes $\lambda_i$. Most notable here is the Katz-Klemm-Vafa formula proven in \cite{PT_KKV}.
Arbitrary linear Hodge integrals with descendents
are better considered as part of the
Pandharipande-Thomas theory of $S \times \BC$.
Their explicit form is taken up in \cite{OS}.
For Gromov--Witten invariants of the K3 surface involving the double ramification cycle, see also \cite{vIOP}.}

For all $k \geq 2$ even, define the weight $k$ Eisenstein series
\[ G_k(q) = - \frac{B_k}{2 \cdot k} + \sum_{n \geq 1} \sum_{d|n} d^{k-1} q^n \]
(with $B_k$ the Bernoulli numbers) and the modular discriminant
\[
\Delta(q) = q \prod_{n \geq 1} (1-q^n)^{24}.
\]

\begin{thm}[Bryan-Leung, {\cite{BL}}] For primitive $\beta$ we have
\[
\left\langle \tau_{0}(\pt)^n \right\rangle^{S}_{\beta} = 
\mathrm{Coeff}_{q^{\beta^2/2}} \left[
\frac{1}{\Delta(q)} \left( q \frac{d}{dq} G_2(q) \right)^n \right].
\]
\end{thm}

Here $\mathrm{Coeff}_{q^m}(f)$ stands for the $q^m$-coefficient of a Laurent series $f$.

\subsection{A conjectural formula for the stationary theory}
The first result of this paper is a conjectural formula
for the stationary primitive invariants of the K3 surface,
which will generalize the Bryan-Leung evaluation.

To state the formula we use the Laurent expansion of the Weierstra{\ss} elliptic function $\wp(z)$
around the origin, which reads
\[ \wp(z) = \frac{1}{z^2} + 2 \sum_{k \geq 4} G_{k}(q) \frac{z^{k-2}}{(k-2)!}, \]
and where we have set $G_k=0$ for $k$ odd. 
For all $k \geq 0$ define the series
\begin{align}
A_k(q) & = \frac{(-1)^k}{(2k+1)!!} \mathrm{Res}_{z=0}\left[ (\wp(z) - 4 G_2)^{k+\frac{1}{2}} \right]  \label{A} \\
B_k(q) & = \frac{(-1)^k}{(2k+3)!!} \mathrm{Res}_{z=0}\left[ (\wp(z) - 4 G_2)^{k+\frac{3}{2}} (\wp + 2 G_2 ) \right] \label{B}
\end{align}
and for $k, \ell \geq 0$ the series
\begin{multline}
C_{k \ell}(q) = \frac{(-1)^{k+\ell-1}}{(2k+1)!! (2\ell+1)!!} \\
\cdot \mathrm{Res}_{z_1=0} \mathrm{Res}_{z_2=0}
\left[ (\wp(z_1) - 4 G_2 )^{k+\frac{1}{2}} (\wp(z_2) - 4 G_2)^{\ell+\frac{1}{2}} ( \wp(z_1-z_2) + 2 G_2 ) \right]. \label{C}
\end{multline}
The Weierstra{\ss} elliptic function $\wp(z)$ is taken above as a formal power series in $z$ with coefficients quasi-modular forms (see Section~\ref{subsec:quasi modular forms}),
so that
\[ \wp(z) - 4 G_2 = \frac{1}{z^{2}} - 4 G_{2} + G_{4}z^{2} + O(z^{4}). \]
Its square root is then computed formally, as in
\[
(\wp - 4 G_2)^{\frac{1}{2}} = \frac{1}{z} - 2 G_{2}z + \left(-2 G_{2}^{2} + \frac{1}{2} G_{4}\right)z^{3} + O(z^{4}).
\]
We obtain the well-defined $z$-series with quasi-modular coefficients:
\[
(\wp - 4 G_2)^{k+\frac{1}{2}} := (\wp(z) - 4 G_2)^k (\wp - 4 G_2)^{\frac{1}{2}}
\]
Taking the residue at $z=0$ means simply taking the $z^{-1}$ coefficient.
The double factorial stands for the product of odd factors:
\[ (2k+1)!! = \frac{(2k+1)!}{2^k \cdot k!} = (2k+1) (2k-1) \cdots 3 \cdot 1. \]
For example,
\begin{gather*}
A_0  =  1, \quad A_1  =  2 G_{2}, \quad A_2  =  2 G_{2}^{2} + \frac{1}{6} G_{4}, \quad \ldots  \\
B_0 = -2 G_{2}^{2} + \frac{5}{6} G_{4}, \quad B_1  =  -\frac{8}{3} G_{2}^{3} + \frac{4}{3} G_{2} G_{4} - \frac{7}{360} G_{6}, \quad \ldots \\
C_{0,0} = 0, \quad C_{1,0} = B_0, \quad C_{1,1} = -\frac{16}{3} G_{2}^{3} + \frac{10}{3} G_{2} G_{4} - \frac{7}{72} G_{6}, \quad \ldots .
\end{gather*}

Given cohomology classes $\gamma_0, \gamma_1, \gamma_2, \ldots \in H^{\ast}(S)$
define the partition function
\begin{align*} Z_{\beta}(\gamma_0, \gamma_1, \ldots ) 
& := \left\langle \exp\left( \sum_{k \geq 0} \tau_k(\gamma_k) \right) \right\rangle^{S}_{g,\beta} \\
& =
\sum_{m_0, m_1, m_2, \ldots \geq 0}
\frac{1}{m_0! m_1! m_2! \cdots }
\Big\langle \tau_{0}(\gamma_0)^{m_0} \tau_{1}(\gamma_1)^{m_1} \tau_2(\gamma_2)^{m_2} \cdots  \Big\rangle^{S}_{\beta}.
\end{align*}
The partition function encodes all Gromov--Witten invariants of $S$ in class $\beta$.\footnote{
Concretely, choose a basis $(e_a)_{a=0}^{23}$ of $H^{\ast}(S)$, let $t_{a,k}$ be formal variables and consider the classes
$\gamma_k = \sum_{a=0}^{23} t_{a,k} e_{a}$. Then we can extract
Gromov--Witten series from the partition function $Z_{\beta}$ by the rule:
\[
\left\langle \tau_{k_1}(e_{a_1}) \cdots \tau_{k_n}(e_{a_n}) \right\rangle^S_{\beta} = 
\left( \frac{d}{d t_{a_1,k_1}} \cdots \frac{d}{dt_{a_n,k_n}} Z_{\beta}(\gamma_1(t), \gamma_2(t), \ldots, ) \right)\Bigg|_{t=0}.
\]}

Let $(\gamma_1, \gamma_2) = \int_{S} \gamma_1 \cup \gamma_2$ denote the intersection pairing on $H^{\ast}(S)$.

The following determines all primitive stationary invariants.

\begin{conj} \label{conj:GW stationary} Let $\beta \in H_2(S,\BZ)$ be primitive and assume that $\deg(\gamma_i) > 0$ for all $i$. Then:
\[ Z_{\beta}(\gamma_0, \gamma_1, \ldots ) = 
\mathrm{Coeff}_{\beta^2/2}\left[ 
\frac{1}{\Delta(q)}
\exp\left( 
\sum_{k \geq 0} (\gamma_k, \beta) A_k(q) 
+ \sum_{k \geq 0} (\gamma_k, 1) B_k(q)
+ \frac{1}{2} \sum_{k,\ell \geq 0} (\gamma_k \cdot \gamma_{\ell}) C_{k \ell}(q) \right)
\right].
\]
\end{conj}
\vspace{8pt}

For example, we obtain the full evaluation of the descendents of point classes:
\[
\left\langle \prod_{i=1}^{n} \tau_{k_i}(\pt) \right\rangle^S_{\beta}
=
\mathrm{Coeff}_{q^{\beta^2/2}} \left[ \frac{B_{k_1}(q) \cdots B_{k_n}(q)}{\Delta(q)} \right].
\]
This generalizes the Bryan-Leung formula by the simple observation:
\[
B_0 = q \frac{d}{dq} G_2 
\]
If $F \in H^2(S)$ is a class satisfying $F^2=0$ and $\beta \cdot F=1$ we obtain
\[
\left\langle \prod_{i=1}^{n} \tau_{k_i}(F) \right\rangle^S_{\beta}
=
\mathrm{Coeff}_{q^{\beta^2/2}} \left[ \frac{A_{k_1}(q) \cdots A_{k_n}(q)}{\Delta(q)} \right].
\]
For $\alpha, \alpha' \in H^2(S)$ with $\alpha \cdot \beta = \alpha' \cdot \beta = 0$ we get:
\[
\blangle \tau_k(\alpha) \tau_{\ell}(\alpha') \brangle^S_{\beta} = \mathrm{Coeff}_{\beta^2/2} (\alpha, \alpha') C_{k \ell}(q) \Delta(q)^{-1}.
\]

The structure of the conjecture is inspired by the structure of multiplicative genera of
 the Hilbert schemes of points of a surface, as given in \cite{EGL}.

\subsection{Descendents of $1$}
For elliptic curves there are explicit rules (called Virasoro constraints) which
recursively remove the descendents insertions $\tau_k(1)$ from the Gromov--Witten bracket \cite{OP3}.
This means that for elliptic curves, the stationary theory determines all Gromov--Witten invariants,
see \cite{Pixton} for explicit formulas.
Virasoro constraints have been conjectured to hold for any smooth projective variety \cite{RahulICM}.
However, for K3 surfaces we use reduced Gromov--Witten theory which is non-standard.
One can check that the usual Virasoro constraints do not hold in this case.
A modified formulation of Virasoro constraints for reduced invariants is not known currently, not even on a conjectural level. Hence currently it is not clear what the dependence on $\tau_k(1)$ factors should be for K3 surfaces.

In this paper we will explain a simple trick that still allows us to gain information
about descendents of $1$ for a large class of invariants. The trick is based on the {\em holomorphic anomaly equation} 
and yields a recursion.
We start with a basic example:

\begin{thm} \label{thm:Example 1} Let $\beta$ be primitive. For $k \geq 2$ we have
\[ \blangle \tau_k(1) \brangle^S_{\beta} = \blangle \tau_{k-3}(\pt) \brangle^S_{\beta} + 2 (k-2) \blangle \tau_{k-1}(F) \brangle^S_{\beta} \]
where $F \in H^2(S)$ is any class such that $\beta \cdot F = 1$ and $F^2=0$.
\end{thm}


A more complicated example is:

\begin{thm} \label{thm:Example 2}
Let $\beta$ be primitive. For $k \geq 2$ and $\ell \geq 0$ we have
\begin{align*}
\left\langle \tau_k(1) \tau_{\ell}(\pt) \right\rangle
& = \blangle \tau_{k - 3}(\pt) \tau_{\ell}(\pt) \brangle 
+ \blangle \tau_{k-2}(\pt) \tau_{\ell-1}(\pt) \brangle  \\
& + (2k+2\ell) \blangle \tau_{k-1}(F) \tau_{\ell}(\pt) \brangle - \blangle \tau_{k-1}(\alpha_1) \tau_{\ell+1}(\alpha_2) \brangle
\end{align*} 
where the class $F \in H^2(S,\BQ)$ is as before and $\alpha_1, \alpha_2 \in H^2(S,\BQ)$ are any classes orthogonal to $\beta$ and $F$ satisfying $\alpha_i^2=0$ and $\alpha_1 \cdot \alpha_2=1$.
\end{thm}

\begin{rmk}
The formulas in Theorem~\ref{thm:Example 1} and~\ref{thm:Example 2} also hold for $k \in \{ 0,1 \}$ if we use the convention $\tau_{k}(\gamma) = \delta_{k+2,0} \int_{S} \gamma$ for all $k<0$. \\
\end{rmk}

An even more general example can be found in Theorem~\ref{thm:more general example of recursion} below.
Our recursion applies to a very large class of descendent invariants,
but does not seem to give nice formulas in general.
Hence for now we just formulate the existence of the recursion
and its scope
and refer to Section~\ref{sec:descendents of 1} for details.

\begin{thm} \label{thm:recursion}
For $\gamma_1, \ldots, \gamma_n \in H^{\ast}(S)$ let $r=|\{ i : \deg_{\BC}(\gamma_i)=1 \}|$. Assume that there are
\[ \alpha_1^{(i)}, \alpha_2^{(i)} \in H^2(S,\BC), \quad i=1,\ldots, 2r \]
such that
\begin{itemize}[itemsep=0pt]
\item each $\alpha_{\ell}^{(i)}$ is orthogonal to $\beta$ and all $\gamma_i$ that lie in $H^2(S)$,
\item $\alpha_1^{(i)} \cdot \alpha_1^{(i)} = \alpha_2^{(i)} \cdot \alpha_2^{(i)} = 0$ and
$\alpha_1^{(i)} \cdot \alpha_2^{(i)} = 1$
\item $\alpha_{\ell}^{(i)} \cdot \alpha_m^{(j)} = 0$ for all $\ell,m$ and $i,j$.
\end{itemize}
Then the recursion described in Section~\ref{sec:descendents of 1}
determines the primitive descendent invariant
$\langle \tau_{k_1}(\gamma_1) \cdots \tau_{k_n}(\gamma_n) \rangle^S_{g,\beta}$
in terms of the stationary invariants.
\end{thm}

For example, if $\gamma_i \in \{ 1, \beta, \pt \}$, then the assumption of Theorem~\ref{thm:recursion} is satisfied for $r \leq 5$.

Together with Conjecture~\ref{conj:GW stationary} this allows to (conjecturally) compute a large class of primitive invariants.
An implementation of this algorithm has been made by the author
and can be found on his webpage.
To give a concrete example, a short computer computation and assuming Conjecture~\ref{conj:GW stationary} gives the genus $29$ invariant
\[
\langle \tau_8(1) \tau_{5}(1) \tau_{10}(1) \tau_4(\pt) \tau_3(\pt) \rangle^S_{g=29, \beta^2/2=3}
=
- \frac{13094491}{333598540006510406597452234752000000}.
\]

\subsection{Polynomial behaviour}
There is a  striking polynomial behaviour
that can be numerically observed in the descendent invariants of the K3 surface.
Let $\beta \in H_2(S,\BZ)$ be a primitive effective curve class with $\beta^2 \neq 0$,
let $\delta_1, \ldots, \delta_t \in H^2(S)$ with $\delta_i \cdot \beta = 0$,
and consider the descendent invariant:
\begin{equation} \label{GWBracket2} \left\langle 
\prod_{i=1}^{r} \tau_{k_i}(1) 
\prod_{i=1}^{s} \tau_{\ell_i}(\beta)
\prod_{i=1}^{t} \tau_{m_i}(\delta_i)
\prod_{i=1}^{u} \tau_{n_i}(\pt)
\right\rangle^S_{\beta}.
\end{equation}

Assuming $k_i, m_i \geq 1$, we can normalize the invariant by a certain combinatorial factor, defining:
\begin{multline*}
\left\llangle 
\prod_{i=1}^{r} \tau_{k_i}(1) 
\prod_{i=1}^{s} \tau_{\ell_i}(\beta)
\prod_{i=1}^{t} \tau_{m_i}(\delta_i)
\prod_{i=1}^{u} \tau_{n_i}(\pt)
\right\rrangle^S_{\beta} \\
:=
\prod_{i=1}^r (-4)^{k_i-1} (2k_i-1)!!
\prod_{i=1}^{s} (-4)^{\ell_i} (2 \ell_i + 1)!!
\prod_{i=1}^{t} (-4)^{m_i-1} (2m_i-1)!!
\prod_{i=1}^{t} (-4)^{n_i} (2n_i+1)!! \\
\cdot
\left\langle 
\prod_{i=1}^{r} \tau_{k_i}(1) 
\prod_{i=1}^{s} \tau_{\ell_i}(\beta)
\prod_{i=1}^{t} \tau_{m_i}(\delta_i)
\prod_{i=1}^{u} \tau_{n_i}(\pt)
\right\rangle^S_{\beta}.
\end{multline*}

We then make the following conjecture:

\begin{conj} \label{conj:polynomiality}
There exists a polynomial $p(x_1, \ldots, x_{r+s+t+u})$ of degree
$\beta^2+2-2u - t + r$ 
such that for all
\begin{equation} k_i \geq \beta^2/2 + 3 - (u+\frac{1}{2} t), \quad \quad \ell_i, m_i, n_i \geq \beta^2/2 +1 - (u+t/2), 
 \label{polynomial range} \end{equation}
we have
\begin{multline*}
\left\llangle 
\prod_{i=1}^{r} \tau_{k_i}(1) 
\prod_{i=1}^{s} \tau_{\ell_i}(\beta)
\prod_{i=1}^{t} \tau_{m_i}(\delta_i)
\prod_{i=1}^{u} \tau_{n_i}(\pt)
\right\rrangle^S_{\beta}
=
p(k_1, \ldots, k_r, \ell_1, \ldots, \ell_s, m_1, \ldots, m_t, n_1, \ldots, n_u).
\end{multline*}
\end{conj}
\vspace{8pt}

In the above conjecture we define the degree of a multivariable polynomial $p(x_1, \ldots, x_n) = 
\sum_{i_1,\ldots,i_n} a_{i_1,\ldots,i_n} x_1^{i_1} \cdots x_n^{i_n}$ to be the maximum of $i_1 + \ldots + i_n$ for which $a_{i_1,\ldots,i_n} \neq 0$.

The bound $\beta^2/2 +1 - (u+t/2)$ is precisely the dimension of the locus of curves in the linear system $|\CO(\beta)|$ incident to $u$ points as well as to generic smooth cycles representing the classes $\delta_i$. This points to a geometric reason for the polynomiality.

\begin{example}
Let $\beta^2 = -2$, so that $\beta$ is the class of a smooth rational curve.
Then all reduced invariants have been computed by Maulik \cite{Maulik}.
The bracket \eqref{GWBracket2} vanishes if $t>0$ or $u>0$ (since the rational curve cannot move).
For the remaining terms the formula is:
\[
\left\llangle 
\prod_{i=1}^{r} \tau_{k_i}(1) 
\prod_{i=1}^{s} \tau_{\ell_i}(\beta)
\right\rrangle^S_{\beta} 
=
(\beta \cdot \beta)^{s} \prod_{p=1}^{r} (2g+p+s-3),
\]
where $k_i \geq 2$ and $\ell_i \geq 0$.
Since $g=\sum_i k_i + \sum_{i} \ell_i - r$,
we find that this matches Conjecture~\ref{conj:polynomiality}.
\end{example}

\begin{rmk}
If some of the $k_i, \ell_i, m_i, n_i$ lie outside the polynomial range (i.e. do not satisfy \eqref{polynomial range}), then we still expect the invariant to be polynomial in those parameters which lie in the polynomial range.
A precise conjecture is given in Section~\ref{subsec:upgraded polynomiality}.
For example, for $k, \ell \geq 1$ we expect in the polynomial range the following:
\[
\llangle \tau_{k}(\pt) \tau_{\ell}(\pt) \rrangle^S_{\beta^2/2=2} = 
8  k^{2} + 8  l^{2} - 12  k - 12  l + 20.
\]
Let $P(k,\ell)$ denote the polynomial on the right. Then for $\ell=0$ (which is outside the polynomial range) one has:
\[
\llangle \tau_{k}(\pt) \tau_{0}(\pt) \rrangle^S_{\beta^2/2=2} = 
8 \, k^{2} - 12 \, k + 28 \ \neq \  P(k,0).
\]
\end{rmk}

The polynomial behaviour of the descendent invariants is a strong constraint on them.
In basic cases it can be used to determines the Gromov--Witten invariants invariants.
We explain this in the case of point insertions:

\begin{thm} \label{thm:Point insertions}
The following statements are equivalent:
\begin{enumerate}[itemsep=0pt]
\item[(i)] For every primitive $\beta$, the series $\langle \tau_k(\pt) \rangle^S_{\beta}$
equals a polynomial of degree $\beta^2$ for $k \geq \beta^2/2$.
\item[(ii)] For every primitive $\beta$, we have the generalized Bryan-Leung formula:
\[
\left\langle \prod_{i=1}^{n} \tau_{k_i}(\pt) \right\rangle^S_{\beta}
=
\mathrm{Coeff}_{q^{\beta^2/2}} \left[ \frac{B_{k_1}(q) \cdots B_{k_n}(q)}{\Delta(q)} \right].
\]
\end{enumerate}
\end{thm}

The proof relies on a characterization of the function $B_k(q)$ in terms of its (partial) polynomial behaviour. Similar characterizations exists for $A_k$ and $C_{k, \ell}$, see Section~\ref{sec:polynomial behaviour}. The partial polynomiality explains the shape of the formulas for $A,B,C$.

The characterization for the function $A_k$ is the simplest. It simply reads:
\begin{thm} \label{prop:Characterization}
The series of functions 
$A_k(q)$, $k \geq 0$ (defined in \eqref{A})
is the unique series of power series satisfying the following conditions:
\begin{enumerate}[itemsep=0pt]
\item[(a)] $A_k(q)$ is a quasi-modular form of weight $2k$ satisfying $\frac{d}{dG_2} A_k = 2 A_{k-1}$ (with $A_{-1}=0$),
\item[(b)] for every $n \geq 0$ there exists a polynomial $p_n(k)$ of degree $2n$ such that
for every $k \geq n$ we have
\[ \Big[ A_k(q) \Big]_{q^n} = 
\frac{1}{(-4)^k (2k+1)!!}
p_n(k), \]
\item[(c)] $A_0 = 1 + O(q)$.
\end{enumerate}
\end{thm}

\subsection{Plan of the paper}
In Section~\ref{sec:Background} we review what is known about the Gromov--Witten theory of K3 surface. In Section~\ref{sec:descendents of 1}
we give the recursion removing descendents of $1$.
In Section~\ref{sec:polynomial behaviour}
we give and prove the characterization of the functions $A_k$, $B_k$, $C_{k \ell}$.
The appendices give further examples of the polynomiality,
and a conjectural Virasoro-type constraint in a special case.

\subsection{Acknowledgements}
I would like to thank Jan-Willem van Ittersum, Aaron Pixton, and Maximilian Schimpf for useful discussions. I also thank the referee for useful suggestions which improved the paper.
The author was supported by the starting grant 'Correspondences in enumerative geometry: Hilbert schemes, K3 surfaces and modular forms', No 101041491
 of the European Research Council.

\section{Background} \label{sec:Background}
We state some background formulae on the Gromov--Witten invariants of the K3 surface.

\subsection{Quasi-modular forms} \label{subsec:quasi modular forms}
The algebra of quasi-modular forms is the free polynomial algebra
\[ \QMod = \BC[G_2, G_4, G_6]. \]
We have $G_k \in \QMod$ for all $k$, and $\QMod$ is graded by weight of the generators.
The differential operators
$D_q := q \frac{d}{dq}$ and $\frac{d}{dG_2}$ act on $\QMod$
and satisfy
\begin{equation} \left[ \frac{d}{dG_2}, D_q \right] = -2 \wt \label{commutation relation} \end{equation}
where the weight operator $\wt \in \End( \QMod)$ acts on the space $\QMod_k$ of quasi-modular forms of weight $k$ by multiplication by $k$.

\subsection{Gromov--Witten theory of K3 surfaces} \label{subsec:GW theory of K3 definitions}
Let $S$ be an algebraic  K3 surface. 
Let $\beta \in H_2(S,\BZ)$ be an effective curve class, i.e. there exists a non-empty algebraic curve $C \subset S$ with $[C] = \beta$.
Let $\Mbar_{g,n}^{\circ}(S,\beta)$ be the moduli space of $n$-marked genus $g$ degree $\beta$ stable maps $(f:C \to S, p_1, \ldots, p_n)$ where the domain $C$ is allowed to be disconnected but with the following assumption: For every connected component $C' \subset C$ we have that (1) the restriction $f|_{C'}$ is non-constant, or (2) the component $C'$ together with the markings incident to $C'$ is stable.

The usual virtual fundamental class of the moduli space $\Mbar_{g,n}^{\circ}(S,\beta)$ vanishes
because of the existence of a holomorphic $2$-form on $S$.
Instead, Gromov--Witten theory is defined by a
reduced virtual fundamental class \cite{KT}
\[
[ \Mbar_{g,n}^{\circ}(S,\beta) ]^{\red}  \in \mathrm{CH}_{g+n}(\Mbar_{g,n}^{\circ}(S,\beta)),
\]
where $\mathrm{CH}_k$ denotes the Chow groups.
The descendent invariants are defined by
\begin{equation} \label{GWinvariant_repeated}
\left\langle \tau_{k_1}(\gamma_1) \cdots \tau_{k_n}(\gamma_n) \right\rangle^{S}_{g,\beta}
=
\int_{[ \Mbar_{g,n}^{\circ}(S,\beta) ]^{\red}} \prod_i \ev_i^{\ast}(\gamma_i) \psi_i^{k_i}.
\end{equation}
where $\psi_i \in H^2(\Mbar_{g,n}^{\circ}(S,\beta))$ are the cotangent line classes,
and $\ev_i : \Mbar_{g,n}^{\circ}(S,\beta) \to S$ are the evaluation maps at the markings.

The integrals \eqref{GWinvariant_repeated} are invariant under deformations of $(S,\beta)$ which preserve the Hodge type of the class $\beta$.
This shows the following result:

\begin{thm} \label{thm:dependence}
The invariant $\left\langle \tau_{k_1}(\gamma_1) \cdots \tau_{k_n}(\gamma_n) \right\rangle^{S}_{g,\beta}$ depends upon $(S,\beta, \gamma_1, \ldots, \gamma_n)$, where $\gamma_i$ are homogeneous, only through following data:
\begin{enumerate}[itemsep=0pt]
\item[(i)] the divisibility $\beta$ and the square $\beta \cdot \beta$,
\item[(ii)] the cohomological degrees of $\gamma_i$,
\item[(iii)] for all $\gamma_i \in H^0(S)$ the degree $\int_S \pt \gamma_i$, and for all $\gamma_i \in H^4(S)$ the degree $\int_S \gamma_i$,
\item[(iv)] for all $i,j$ with $\deg_{\BC}(\gamma_i)=1$ the pairings
$\beta \cdot \gamma_i$ and $\gamma_i \cdot \gamma_j$.
\end{enumerate}
In other words,
if 
$(S',\beta',\gamma_1', \ldots, \gamma_n')$ has the same data as $(S,\beta, \gamma_1, \ldots, \gamma_n)$, then
\[
\left\langle \tau_{k_1}(\gamma_1) \cdots \tau_{k_n}(\gamma_n) \right\rangle^{S}_{g,\beta}
=\left\langle \tau_{k_1}(\gamma_1') \cdots \tau_{k_n}(\gamma_n') \right\rangle^{S'}_{g,\beta'}.
\]
\end{thm}
\begin{proof}
This was proven first by Buelles \cite{BuellesK3}, see also \cite[Sec.2.2]{GWNL_HK}. We sketch the argument:
By the moduli theory of K3 surfaces there exists a deformation from $(S',\beta')$ to $(S,\beta)$ that keeps $\beta'$ of Hodge type. Hence by deformation invariance, we can assume that $S'=S$ and $\beta'=\beta$.
Moreover, any isometry $m : H^{\ast}(S,\BZ) \to H^{\ast}(S,\BZ)$ which preserves $\beta$ and the K\"ahler cone, can be realized by a deformation of the K3 surface which preserves the Hodge type of $\beta$. Thus again by deformation invariance we have:
\[
\left\langle \tau_{k_1}(m(\gamma_1)) \cdots \tau_{k_n}(m(\gamma_n)) \right\rangle^{S}_{g,\beta}
=
\left\langle \tau_{k_1}(\gamma_1) \cdots \tau_{k_n}(\gamma_n) \right\rangle^{S}_{g,\beta}.
\]
The group of such isometries $m$ is Zariski dense in the group $O(H^2(S,\BC))_{\beta}$, by which we denote the stabilizer of $\beta$ in the complex orthogonal group $O(H^2(S,\BC))$. It follows:
\begin{equation} \label{Dfsdf0-9}
\ev_{\ast}\left( \prod_{i=1}^{n} \psi_i^{k_i} \cap [\Mbar^{\circ}_{g,n}(S,\beta)]^{\red} \right)
\in H^{\ast}(S^n)^{O(H^2(S,\BC))_{\beta}}.
\end{equation}
The invariants of the orthogonal group are well-understand. Concretely, one has that the ring
$H^{\ast}(S^n)^{O(H^2(S,\BC))_{\beta}}$ is generated by the pullbacks of the classes $\beta$ and $\pt$ from factors, and the big diagonals $\Delta_{ij}$.
It follows that if $(\gamma_1, \ldots, \gamma_n)$ and $(\gamma_1', \ldots, \gamma_n')$ have the same pairing data, their intersection with \eqref{Dfsdf0-9} is the same. This finishes the proof.
\end{proof}

For convenience we 
state the divisor, string and dilaton equation.

\begin{lemma} \label{lemma:divisor string dilaton}
For any effective $\beta \in H_2(S,\BZ)$ and class $D \in H^2(S)$ we have
\begin{align*}
& \left\langle \tau_{k_1}(\gamma_1) \cdots \tau_{k_n}(\gamma_n) \tau_0(D) \right\rangle^{S}_{g,\beta} \\
= & (\beta \cdot D) \left\langle \tau_{k_1}(\gamma_1) \cdots \tau_{k_n}(\gamma_n) \right\rangle^{S}_{g,\beta} \\
& + \sum_{i=1}^{n} \left\langle \tau_{k_1}(\gamma_1) \cdots \tau_{k_i-1}(\gamma_i D ) \cdots \tau_{k_n}(\gamma_n) \right\rangle^{S}_{g,\beta}  \\
& + \sum_{\substack{1 \leq i < j \leq n \\ k_i = k_j = 0}} 
\left( \int_S \gamma_i \gamma_j D \right) \left\langle \tau_{k_1}(\gamma_1) \cdots \widehat{\tau_{k_i}(\gamma_i)} \cdots \widehat{\tau_{k_j}(\gamma_j)} \cdots  \tau_{k_n}(\gamma_n)  \right\rangle^{S}_{g,\beta}
\end{align*}
and
\begin{gather*}
\left\langle \tau_{k_1}(\gamma_1) \cdots \tau_{k_n}(\gamma_n) \tau_0(1) \right\rangle^{S}_{g,\beta}
= \sum_{i=1}^{n} \left\langle \tau_{k_1}(\gamma_1) \cdots \tau_{k_i-1}(\gamma_i ) \cdots \tau_{k_n}(\gamma_n) \right\rangle^{S}_{g,\beta}  \\
+ \sum_{\substack{1 \leq i < j \leq n \\ k_i = k_j = 0}} 
\left( \int_S \gamma_i \gamma_j \right) \left\langle \tau_{k_1}(\gamma_1) \cdots \widehat{\tau_{k_i}(\gamma_i)} \cdots \widehat{\tau_{k_j}(\gamma_j)} \cdots  \tau_{k_n}(\gamma_n)  \right\rangle^{S}_{g,\beta} 
\end{gather*}
and
\[ \left\langle \tau_{k_1}(\gamma_1) \cdots \tau_{k_n}(\gamma_n) \tau_1(1) \right\rangle^{S}_{g,\beta}
=
(2g-1+n) \left\langle \tau_{k_1}(\gamma_1) \cdots \tau_{k_n}(\gamma_n) \right\rangle^{S}_{g,\beta}. \]
\end{lemma}

\begin{proof}
This follows by the usual arguments.
The non-standard formulation arises because we use moduli space of stable maps with disconnected domain and the marking that we consider can lie on a component which becomes unstable if we would forget the marking. This components yield the extra contributions. Then we use the evaluation of the (usual) virtual class
of the moduli space of connected stable maps $\Mbar_{g,n}(S,\beta)$ in degree zero:
\begin{equation} \label{vir class}
[ \Mbar_{g,n}(S,0) ]^{\text{vir}}
=
\begin{cases}
[\Mbar_{0,n} \times S] & \text{if } g= 0 \\
\mathrm{pr}_2^{\ast} c_2(S) \cap [\Mbar_{1,n}  \times S]
& \text{if } g= 1\\
0 & \text{if } g \geq 2
\end{cases}
\end{equation}
\end{proof}

\begin{rmk}{\em (On the relationship between connected and disconnected invariants; this remark may be skipped)}
We can also define {\em connected} Gromov--Witten invariants
\[
\left\langle \tau_{k_1}(\gamma_1) \cdots \tau_{k_n}(\gamma_n) \right\rangle^{S,\mathrm{connected}}_{g,\beta}
:=
\int_{[ \Mbar_{g,n}(S,\beta) ]^{\red}} \prod_i \ev_i^{\ast}(\gamma_i) \psi_i^{k_i},
\]
by integrating over the moduli space of stable maps $f: C \to S$
with connected domain curve $C$.
The relationship to the disconnected invariants is as follows:
\begin{lemma}
If $k_i \geq 2-\deg_{\BC}(\gamma_i)$ for all $i$ (that is, the divisor, string and dilaton equation can not be applied), then the connected and disconnected invariants coincide:
\[
\left\langle \tau_{k_1}(\gamma_1) \cdots \tau_{k_n}(\gamma_n) \right\rangle^{S}_{g,\beta}
=
\left\langle \tau_{k_1}(\gamma_1) \cdots \tau_{k_n}(\gamma_n) \right\rangle^{S,\mathrm{connected}}_{g,\beta}
\]
\end{lemma}
\begin{proof}
The reduced virtual class of $\Mbar_{g,n}^{\circ}(S,\beta)$ vanishes on all components which parametrize maps $f:C \to S$ where $f$ is non-constant on more than one connected component,
because then the standard obstruction theory has two trivial summands, but only one is removed by the reduction procedure.
Hence the only contributing components to the left side are
\[ \Mbar_{g',n'}(S,\beta) \times \prod_{i} \Mbar_{g_i, n_i}(S,0) \]
for some $g',n',g_i,n_i$. The reduced virtual class of this component is
\[ [ \Mbar_{g',n'}(S,\beta) ]^{\red} \times \prod_{i=1}^{\ell} [\Mbar_{g_i, n_i}(S,0)]^{\vir}. \]
By \eqref{vir class} one sees that this component contributes zero if $\ell \geq 1$.
\end{proof}
\end{rmk}

\subsection{Elliptic K3 surfaces and generating series} \label{subsec:elliptic K3}
As discussed in the last section, in order to evaluate primitive invariants we can specialize to any K3 surface that we like, as long as it has primitive curve classes of arbitrary square. The most useful K3 surface to choose is an elliptically fibered one.

Hence let $S \to \p^1$ be an elliptic K3 surface with section $B$ and fiber class $F$. Let also $W = B+F$. This choice is made so that $W,F$ span the standard hyperbolic lattice:
\[ W^2 = 0, \quad W \cdot F = 1, \quad F^2=0. \]
We define the multilinear bracket
\[ \left\langle \tau_{k_1}(\gamma_1) \cdots \tau_{k_n}(\gamma_n) \right\rangle^{\GW}
:=
\sum_{h=0}^{\infty} q^{h-1} \left\langle \tau_{k_1}(\gamma_1) \cdots \tau_{k_n}(\gamma_n) \right\rangle^S_{g, B+hF} \]
where on the right 
the genus is specified by the dimension constraint.

If $\gamma \in H^{\ast}(S)$ is an eigenvectors of the operator,
\[ [ B \cdot (-) , \pi^{\ast} \pi_{\ast} ] : H^{\ast}(S) \to H^{\ast}(S) \]
then we let $\wt(\gamma)$ be its eigenvalue. Concretely,
\[
\wt(\gamma) =
\begin{cases}
1 & \text{ if } \gamma \in \{ \pt, W \} \\
-1 & \text{ if } \gamma \in \{ 1, F \} \\
0 & \text{ if } \gamma \perp \{ \pt, 1, W, F \}.
\end{cases}
\]
We assume below that all $\gamma_i$ are homogeneous with respect to this grading.

Recall the following two basic results:

\begin{thm}[{Quasimodularity, \cite{MPT} \text{together with} \cite[Sec.4.6]{BOPY}}] 
\label{thm:quasimodularity}
We have that
\[
\left\langle \tau_{k_1}(\gamma_1) \cdots \tau_{k_n}(\gamma_n) \right\rangle^{\GW} \in \frac{1}{\Delta(q)} \mathsf{QMod}_{2g+n+\sum_i \wt(\gamma_i)}.
\]
\end{thm}

\begin{rmk}
Alternatively, by the dimension constraint the weight is given by
\[ 2g + n + \sum_i \wt(\gamma_i) = \sum_{i} (2 k_i + 2 \deg_{\BC}(\gamma_i) + \wt(\gamma_i) - 1). \]
\end{rmk}


\begin{thm}[{Holomorphic anomaly equation \cite{HAE}}] \label{thm:HAE}
\begin{align*}
\frac{d}{dG_2} \left\langle \tau_{k_1}(\gamma_1) \cdots \tau_{k_n}(\gamma_n) \right\rangle^{\GW}
& = 2 \left\langle \tau_{k_1}(\gamma_1) \cdots \tau_{k_n}(\gamma_n) \tau_0(1) \tau_0(F) \right\rangle^{\GW} \\
& -2 \sum_{i=1}^{n} \left\langle \tau_{k_1}(\gamma_1) \cdots \tau_{k_i+1}(\pi^{\ast} \pi_{\ast} \gamma_i ) \cdots \tau_{k_n}(\gamma_n) \right\rangle^{\GW} \\
& + 20 \sum_{i=1}^{n} (\gamma_i, F) \left\langle \tau_{k_1}(\gamma_1) \cdots \tau_{k_i}(F) \cdots \tau_{k_n}(\gamma_n) \right\rangle^{\GW} \\
& -2 \sum_{i<j} \left\langle \tau_{k_1}(\gamma_1) \cdots \tau_{k_i}(\sigma_1(\gamma_i, \gamma_j)) \cdots \tau_{k_j}( \sigma_2(\gamma_i, \gamma_j)) \cdots \tau_{k_n}(\gamma_n) \right\rangle^{\GW}
\end{align*}
where $\sigma: H^{\ast}(S^2) \to H^{\ast}(S^2)$ is defined by
\[
\sigma(\gamma \boxtimes \gamma') = 0
\ \text{ whenever } \gamma \text{ or } \gamma' \text{ lie in }
H^0(S) \oplus \BQ F  \oplus H^4(S),
\]
and by
\begin{alignat*}{2}
\sigma( W \boxtimes W ) & = \Delta_{V} , \quad \quad \quad
& \sigma( W \boxtimes \alpha ) & = - \alpha \boxtimes F, \\
\sigma( \alpha \boxtimes W ) & = - F \boxtimes \alpha,
& \sigma( \alpha, \alpha' ) & = ( \alpha, \alpha' ) F \boxtimes F
\end{alignat*}
for all $\alpha, \alpha' \in V := \{ W, F \}^{\perp} \subset H^2(S)$.
\end{thm}

We also recall a basic splitting statement.
Define the normalized correlators:
\begin{align*}
\left\langle \tau_{k_1}(\gamma_1) \cdots \tau_{k_n}(\gamma_n) \right\rangle^{\GW, \prime}
& =
\frac{ \left\langle \tau_{k_1}(\gamma_1) \cdots \tau_{k_n}(\gamma_n) \right\rangle^{\GW} }{ \left\langle 1 \right\rangle^{\GW} } \\
& = \Delta(q) \left\langle \tau_{k_1}(\gamma_1) \cdots \tau_{k_n}(\gamma_n) \right\rangle^{\GW}.
\end{align*}

The standard degeneration argument given in \cite{MPT} yields:

\begin{lemma} \label{lemma:splitting formula}
Assume that $\gamma_i \in \{ F, \pt \}$ for all $i$, and $\alpha_j \in \{ B, F \}^{\perp} \subset H^2(S)$. Then
\[ \left\langle \prod_{i=1}^{n} \tau_{k_i}(\gamma_i) \prod_{j=1}^{n'} \tau_{\ell_j}(\alpha_j)  \right\rangle^{\GW, \prime}
=
\prod_{i=1}^{n} \left\langle \tau_{k_i}(\gamma_i) \right\rangle^{\GW, \prime} \cdot \left\langle \prod_{j=1}^{n'} \tau_{\ell_j}(\alpha_j)  \right\rangle^{\GW, \prime}. \]
\end{lemma}

\subsection{Two more remarks}
We end this section with two more remarks, which will not be used later on.
\begin{rmk}
We can restate Conjecture~\ref{conj:GW stationary} in the language of elliptic K3 surfaces as follows. Define
\[ Z(\gamma_0, \gamma_1, \ldots ) 
:= \left\langle \exp\left( \sum_{k \geq 0} \tau_k(\gamma_k) \right) \right\rangle^{\GW}. \]
Then Conjecture~\ref{conj:GW stationary} is equivalent the statement: for all $\gamma_i$ with $\deg_{\BC}(\gamma_i)>0$ we have
\[ Z(\gamma_0, \gamma_1, \ldots ) = 
\frac{1}{\Delta(q)}
\exp\left( \sum_{k \geq 0} (\gamma_k, W + D_{q} F ) A_k(q)
+ \sum_{k \geq 0} (\gamma_k, 1) B_k(q)
+ \frac{1}{2} \sum_{k,\ell \geq 0} (\gamma_k \cdot \gamma_{\ell}) C_{k \ell}(q) \right)
\]
where the $D_{q}$'s in the formula above stand for commuting the operators to the left and
applying them to the full series. 
\end{rmk}

\begin{rmk}
By the degeneration argument of \cite{MPT} we obtain relatively nice expressions for the two most basic invariants
\[
\langle \tau_k(F) \rangle^{\GW}, \quad \langle \tau_k(\pt) \rangle^{\GW}
\]
This was already used in \cite[App.B]{MPT}. Explicit formulas can be found in \cite[Appendix]{K3xP1}.
Let us recall the expression in the first case:

Consider for $g \geq 0$ (with obvious notation) the following generating series of Gromov--Witten invariants of an elliptic curve $E$,
\[
L_{g,k}
= 
\begin{cases}
\left\langle \BE^{\vee}(1) \frac{\ev_1^{\ast}(\pt)}{1-\psi_1}  \ev_2^{\ast}(1) \psi_2^k \right\rangle^{E}_{g} & g \geq 1 \\
(-1)^k & g = 0,
\end{cases}
\]
where we used the inverse Hodge class
\[ \BE^{\vee}(1) = 1 - \lambda_1 + \lambda_2 + \ldots + (-1)^g \lambda_g. \]
Then one has:
\[ \blangle \tau_k(F) \brangle^{\GW} = \frac{1}{\Delta(q)} \sum_{g=0}^k L_{g,k} \frac{ (-G_2)^{k-g}}{(k-g)!}. \]
The right hand side can be effectively computed using the methods of \cite{OP}.
However, it is unclear to the author how to deduce the corresponding case of Conjecture~\ref{conj:GW stationary} from this.
\end{rmk}

\section{Descendents of $1$} \label{sec:descendents of 1}
We present here the new recursion that allows one to compute a large class of descendent invariants effectively in terms of the stationary theory.

\subsection{Removing insertions with $W$}
Let $S \to \p^1$ be an elliptic K3 surface as in Section~\ref{subsec:elliptic K3}
and set
\[ V = \{ x \in H^2(S,\BC)\, | \, x \cdot W = x \cdot F = 0 \}. \]
For $\delta_i \in V$ arbitrary classes and $d \geq 1$ consider a general descendent invariant
\[
\left\langle
\prod_{i=1}^{r} \tau_{k_i}(1) 
\prod_{i=1}^{s} \tau_{\ell_i}(F)
\prod_{i=1}^{s'} \tau_{\ell'_i}(W)
\prod_{i=1}^{t} \tau_{m_i}(\delta_i)
\prod_{i=1}^{u} \tau_{n_i}(\pt)
\right\rangle^S_{\beta=W+dF}.
\]

Our first observation is that whenever there are not too many $\delta_i$, one can get rid of the factors $\tau_{\ell}(W)$.
\begin{prop} \label{prop:Removing W}
Assume that there are $\alpha_1, \alpha_2 \in V$ such that
\[ \alpha_1^2 = \alpha_2^2 = 0, \quad \alpha_1 \cdot \alpha_2 = 1, \quad \forall i:\ \alpha_1 \cdot \delta_i = \alpha_2 \cdot \delta_i = 0. \]
Then we have
\begin{gather*}
\left\langle
\prod_{i=1}^{r} \tau_{k_i}(1) 
\prod_{i=1}^{s} \tau_{\ell_i}(F)
\prod_{i=1}^{s'} \tau_{\ell'_i}(W)
\prod_{i=1}^{t} \tau_{m_i}(\delta_i)
\prod_{i=1}^{u} \tau_{n_i}(\pt)
\right\rangle^S_{W+dF} \\
=
\left\langle
\prod_{i=1}^{r} \tau_{k_i}(1) 
\prod_{i=1}^{s} \tau_{\ell_i}(F+\alpha_2)
\prod_{i=1}^{s'} \tau_{\ell'_i}(dF+\alpha_1)
\prod_{i=1}^{t} \tau_{m_i}(\delta_i)
\prod_{i=1}^{u} \tau_{n_i}(\pt)
\right\rangle^S_{W+dF}.
\end{gather*}
\end{prop}
\begin{proof}
The set of classes
\[ (W+dF, W, F, \delta_i), \quad (W+dF, dF + \alpha_1, F+\alpha, \delta_i) \]
have the same intersection pairings. Hence the claim follows from Proposition~\ref{thm:dependence}.
(What is curious is that there is no isometry that can send $W+dF, W,F$ to $W+dF, dF+\alpha_1, F+\alpha_2$, since the second set of vectors is linearly independent, while the first one is not. Nevertheless, we still have an equality of Gromov--Witten invariants.)
\end{proof}

\subsection{The main recursion step}
Since we have Proposition~\ref{prop:Removing W},
let us consider an invariant where there are no $W$-factors:
\[
I = \left\langle
\prod_{i=1}^{r} \tau_{k_i}(1) 
\prod_{i=1}^{s} \tau_{\ell_i}(F)
\prod_{i=1}^{t} \tau_{m_i}(\delta_i)
\prod_{i=1}^{u} \tau_{n_i}(\pt)
\right\rangle^{\GW},
\]
where $\delta_i \in V$. By the string and dilaton equation we can assume that $k_i \geq 2$ for all $i$.

Assume that there are classes $\alpha_1, \alpha_2 \in H^2(S,\BC)$ such that
\[ \alpha_1^2 = \alpha_2^2 = 0, \quad \alpha_1 \cdot \alpha_2 = 1, \quad \forall i:\ \alpha_1 \cdot \delta_i = \alpha_2 \cdot \delta_i = 0. \]
We give a formula for $I$ that involves only $(r-1)$ many factors of $\tau_k(1)$.
Consider the modified invariant
\[
I_W = 
\left\langle
\tau_{k_1}(W)
\prod_{i=2}^{r} \tau_{k_i}(1) 
\prod_{i=1}^{s} \tau_{\ell_i}(F)
\prod_{i=1}^{t} \tau_{m_i}(\delta_i)
\prod_{i=1}^{u} \tau_{n_i}(\pt)
\right\rangle^{\GW}.
\]

\begin{lemma} \label{lemma:path 1} We have
\begin{equation} \label{fwe333_2}
\frac{d}{dG_2} I_W
=
-2 I + (...)
\end{equation}
where $(...)$ stands for terms involving invariants with $\leq (r-1)$ factors of $\tau_{k}(1)$.
\end{lemma}
\begin{proof}
This follows immediately from applying the holomorphic anomaly equation (Theorem~\ref{thm:HAE}) to $I_W$. Here the first term on the right side of the holomorphic anomaly equation can be reduced by the string equation (Lemma~\ref{lemma:divisor string dilaton}) to involve only $r-1$ factors of $\tau_k(1)$. The second term for $i=1$ yields the term $-2I$ since $\pi^{\ast} \pi_{\ast}(W) = 1$. The other terms do not create any new $\tau_k(1)$ factors.
\end{proof}

On the other hand, we can also first apply Proposition~\ref{prop:Removing W} to $I_W$ before applying the holomorphic anomaly equation.
First, by Proposition~\ref{prop:Removing W} we have:
\begin{align*}
I_W = &
D_{q}\left( \left\langle
\tau_{k_1}(F)
\prod_{i=2}^{r} \tau_{k_i}(1) 
\prod_{i=1}^{s} \tau_{\ell_i}(F)
\prod_{i=1}^{t} \tau_{m_i}(\delta_i)
\prod_{i=1}^{u} \tau_{n_i}(\pt)
\right\rangle^{\GW} \right) \\
& + \left\langle
\tau_{k_1}(\alpha_1)
\prod_{i=2}^{r} \tau_{k_i}(1) 
\prod_{i=1}^{s} \tau_{\ell_i}(F + \alpha_2)
\prod_{i=1}^{t} \tau_{m_i}(\delta_i)
\prod_{i=1}^{u} \tau_{n_i}(\pt)
\right\rangle^{\GW} 
\end{align*}
(where we don't have any $\alpha_2$ in the first line by applying Theorem~\ref{thm:dependence}).
Next we apply $\frac{d}{dG_2}$. Using the commutation relation \eqref{commutation relation}
we obtain three terms:
\begin{equation} \label{fwe333}
\begin{aligned}
\frac{d}{dG_2} I_W = &
-2 \wt \left( \left\langle
\tau_{k_1}(F)
\prod_{i=2}^{r} \tau_{k_i}(1) 
\prod_{i=1}^{s} \tau_{\ell_i}(F)
\prod_{i=1}^{t} \tau_{m_i}(\delta_i)
\prod_{i=1}^{u} \tau_{n_i}(\pt)
\right\rangle^{\GW} \right) \\
& + 
D_{q} \frac{d}{dG_2} \left\langle
\tau_{k_1}(F)
\prod_{i=2}^{r} \tau_{k_i}(1) 
\prod_{i=1}^{s} \tau_{\ell_i}(F)
\prod_{i=1}^{t} \tau_{m_i}(\delta_i)
\prod_{i=1}^{u} \tau_{n_i}(\pt)
\right\rangle^{\GW} \\
& +
\frac{d}{dG_2}
\left\langle
\tau_{k_1}(\alpha_1)
\prod_{i=2}^{r} \tau_{k_i}(1) 
\prod_{i=1}^{s} \tau_{\ell_i}(F + \alpha_2)
\prod_{i=1}^{t} \tau_{m_i}(\delta_i)
\prod_{i=1}^{u} \tau_{n_i}(\pt)
\right\rangle^{\GW}.
\end{aligned}
\end{equation}
The $\frac{d}{dG_2}$-derivative can be evaluated by the holomorphic anomaly equation plus the string equation. The outcome is that all terms only involve at most $(r-1)$ descendents of $1$. (Since the only way new descendents of $1$ appear is via the $\pi^{\ast} \pi_{\ast}(\gamma_i)$ term, hence from cohomology insertions $\gamma_i=W$).

We now simply equate \eqref{fwe333} with \eqref{fwe333_2} and solve for $I$. This gives:
\begin{prop} \label{prop:Removing 1}
Assume that there are classes $\alpha_1, \alpha_2 \in H^2(S,\BC)$ such that
\[ \alpha_1^2 = \alpha_2^2 = 0, \quad \alpha_1 \cdot \alpha_2 = 1, \quad \forall i:\ \alpha_1 \cdot \delta_i = \alpha_2 \cdot \delta_i = 0. \]
Let (...) stand for the same term as in Lemma~\ref{lemma:path 1}. Then
\[
-2 I \ =\  (\textup{Right hand side of } \eqref{fwe333})
\ +\  (...).
\]
In particular,
$I$ can be expressed as a sum of invariants
involving $\leq (r-1)$ factors $\tau_{k}(1)$.
\end{prop}

We get our general recursion by first applying Proposition~\ref{prop:Removing W}, then applying Proposition~\ref{prop:Removing 1}, then repeating the process. In each step we have one less factor of $\tau_k(1)$. After finitely many steps we are left with evaluating a stationary invariant.
This proves Theorem~\ref{thm:recursion} (to go from the general K3 surface to the elliptic one, one uses Theorem~\ref{thm:dependence}).

\subsection{An example of the recursion}
The above recursion can be made explicit in a very basic, but still somewhat general case.
We omit the somewhat tedious computation and just state the result.

Let $S \to \p^1$ be the elliptic K3 and consider a orthogonal decomposition
\[ H^2(S,\BZ) = U_1 \oplus U_2 \oplus L \]
where $U_1 = \mathrm{Span}(W,F)$ and $U_2 = \mathrm{Span}(\alpha_1, \alpha_2)$ with
\[ \alpha_1^2 = \alpha_2^2 = 0, \alpha_1 \cdot \alpha_2 = 1. \]

\begin{thm} \label{thm:more general example of recursion}
For any $\gamma_1, \ldots, \gamma_n \in \{ F, \pt, L \}$ and $k \geq 2$ we have
\begin{gather*}
\blangle \tau_{k}(1) \tau_{k_1}(\gamma_1) \cdots \tau_{k_n}(\gamma_n) \brangle^{\GW}
= \\
\left( 2k-4 -n + |\{ i : \gamma_i = \pt \}| + 2\sum_{i=1}^{n} (k_i + \deg_{\BC}(\gamma_i)) \right)
\blangle \tau_{k-1}(F) \tau_{k_1}(\gamma_1) \cdots \tau_{k_n}(\gamma_n) \brangle^{\GW} \\
+ \blangle \tau_{k-2}(\pt) \tau_0(1) \prod_{i=1}^{n} \tau_{k_i}(\gamma_i) \brangle^{\GW}  \\
- \sum_{i: \gamma_i = \pt} \blangle
\tau_{k-1}(\alpha_1) \tau_{k_i+1}(\alpha_2)
\prod_{j \neq i} \tau_{k_j}(\gamma_j) \brangle^{\GW} \\
+ \sum_{i: \gamma_i \in L} \blangle \tau_{k_i}(F) \tau_{k-1}(\gamma_i) \prod_{j \neq i} \tau_{k_j}(\gamma_j) \brangle^{\GW} \\
- \sum_{i \neq j : \gamma_i, \gamma_j \in L} (\gamma_i, \gamma_j) \blangle \tau_{k-1}(\alpha_1)
\tau_{k_i}(\alpha_2) \tau_{k_j}(F)
\prod_{\ell \neq i,j} \tau_{k_{\ell}}(\gamma_{\ell})
\brangle^{\GW}.
\end{gather*}
\end{thm}

\section{Polynomial behaviour} \label{sec:polynomial behaviour}
The main goal of this section is to characterize the functions $A_k, B_k, C_{k \ell}$ defined in \eqref{A}, \eqref{B}, \eqref{C} in terms of some basic qualitative properties.
This will provide the connection between Conjecture~\ref{conj:GW stationary} (which evaluates explicitly the stationary theory)
and Conjecture~\ref{conj:polynomiality} (which describes the polynomial structure).

\subsection{Characterization of the $A$ series}
Recall the definition of the $A$-series:
\begin{equation} \label{AA} A_k(q) = \frac{(-1)^k}{(2k+1)!!} \mathrm{Res}_{z=0}\left[ (\wp(z) - 4 G_2)^{k+\frac{1}{2}} \right]. \end{equation}
We want to prove here the following characterization:
\begin{thm} \label{thm:Characterization2}
The series of functions 
$A_k(q)$, $k \geq 0$ (defined in \eqref{A})
is the unique series of power series satisfying the following conditions:
\begin{enumerate}[itemsep=0pt]
\item[(a)] $A_k(q)$ is a quasi-modular form of weight $2k$ satisfying $\frac{d}{dG_2} A_k = 2 A_{k-1}$ (with $A_{-1}=0$),
\item[(b)] For every $n \geq 0$ there exists a polynomial $p_n(k)$ of degree $2n$ such that
for every $k \geq n$ we have
\[ \Big[ A_k(q) \Big]_{q^n} = 
\frac{1}{(-4)^k (2k+1)!!}
p_n(k), \]
\item[(c)] $A_0 = 1 + O(q)$
\end{enumerate}
\end{thm}

We will first prove the uniqueness part of Theorem~\ref{thm:Characterization2}:

\begin{lemma} \label{lemma:uniqueness}
There is at most one series $A_k(q)$ for $k \geq 0$ satisfying conditions (a-c) of
Theorem~\ref{thm:Characterization2}.
\end{lemma}

To prove the lemma we will use the following well-known fact:
\begin{lemma} \label{Modconstraint}
Let $f \in \Mod_k$.
If $[f(q)]_{q^{\ell}} = 0$ for all
$\ell \leq \lfloor \frac{k}{12} \rfloor$,
then $f(q) = 0$.
\end{lemma}

\begin{proof}[Proof of Lemma~\ref{lemma:uniqueness}]
We argue by induction on $k$ that $A_0, \ldots, A_{k-1}$ uniquely determine $A_k$.
By (a) and (c) we have $A_0=1$.
Hence assume that $A_{i}$ is known for $i \leq k-1$.
By condition (a) and Lemma~\ref{Modconstraint} it is enough to know the first $\lfloor k/12 \rfloor$ Fourier-coefficients of $A_k$.
Hence let $n \leq k/12$ and let us find $[A_k(q)]_{q^n}$.
By (b) and since $k \geq n$ we have 
\[ [A_k(q)]_{q^n} = \frac{1}{(-4)^k (2k+1)!!} p_n(k). \]
The polynomial $p_n(k)$ is of degree $2n \leq k/6$.
By induction we know the value of $p_n$ for all $k'$ where
$n \leq k' < k$, in particular for all $k/12 \leq k' < k$.
Hence $p_n$ is uniquely determined already by these values, and hence so is $p_n(k)$.
%
\end{proof}

Next we prove that $A_k(q)$ defined by \eqref{AA}
satisfies the conditions (a-c). For (a) and (c) this is easy by considering the $z$-expansion of $\wp(z)$, and left to the reader. 
We need to prove (b).
We do this in two parts. The first part is also not difficult:

Let $p=e^{z}$ and consider the Fourier expansion of the Weierstrass elliptic function:\footnote{To distinguish between the Fourier expansion $\wp(p,q)$ and the $z$-series $\wp(z)$ we will always include the $q$-dependence in the former in the notation.}
\[
\wp(p,q) = \frac{1}{12} + \frac{p}{(1-p)^2} + \sum_{d \geq 1} \sum_{k | d} k (p^k - 2 + p^{-k}) q^{d}.
\] 
We obtain that:
\[
\wp(p,q) - 4 G_2(q) = \frac{1}{4} \frac{(p+1)^2}{ (p-1)^2 } + (p - 6 + p^{-1}) q + \ldots .
\]
Consider now the square root, taken formally as a power series in $q$ with coefficients Laurent series
(this is possible since the $q^0$-coefficient is a square):
\begin{equation} \sqrt{ \wp(p,q) - 4G_2(q) } = \frac{1}{2} \frac{(p+1)}{(p-1)}
+ \frac{(p-1) (p-6+p^{-1})}{(p+1)} q + a_2(p) q^2 + a_3(p) q^3 + \ldots \label{expansion_g} \end{equation}
where $a_n(p) = p^{-n} b_n(p)/(p+1)^{2n-1}$ for some polynomial $b_n(p) \in \BQ[p]$.

\begin{prop} \label{lemma:p0 coefficient}
There exists polynomials $p_n(k)$ of degree $2n$ such that for all $k$:
\[ \left[ \sqrt{ \wp(p,q) - 4G_2(q) }^{2k+1} \right]_{p^0} = \frac{1}{2^{2k+1}}\sum_{n \geq 0} p_n(k) q^n. \]
\end{prop}
\begin{proof}
The series $F(p,q) = 4 (\wp(p,w) - 4 G_2(q))$ and $H(p,q) = 2\sqrt{\wp(p,q) - 4 G_2(q)}$ are both of the form
\[ (1 + a_{01} p + a_{02} p^2 + \ldots ) + (a_{1,-1} p^{-1} + a_{10} + \ldots ) q + ( a_{2,-2} p^{-2} + a_{2,-1} p^{-1} + a_{20} + \ldots ) q^2 + \ldots \]
for some $a_{ij} \in \BQ$.
Hence $F(p,qp)$ and $H(p,qp)$ lie in $\BC[[p,q]]$ and the result follows from
Lemma~\ref{lemma:combinatorics}(c) below.
\end{proof}
\begin{lemma} \label{lemma:combinatorics}
\begin{enumerate}
	\item[(a)] Let $f \in 1+R[[q]]$ be a power series with constant term one and coefficients in a ring $R$.
	Then the $q^n$-coefficient of $f^k$ is polynomial of degree $\leq n$ in $k$ with coefficients in $R$.
	\item[(b)] Let $f=\sum_{i,j \geq 0} a_{ij} p^i q^j$ and $h=\sum_{i,j \geq 0} h_{ij} p^i q^j$ be power series in $p,q$ with $a_{00}=1$. Then the coefficient
	$[ f^k h ]_{p^m q^n}$ is a polynomial in $k$ of degree $\leq m+n$. If $h_{00}=0$, then it is of degree $\leq m+n-1$.
\end{enumerate}
\end{lemma}
\begin{proof}
For (a) by the binomial theorem we can write $f^k = \sum_{i \geq 0} \binom{k}{i} q^i (1 + a_2 q^1 + a_3 q^2 + \ldots )^i$. Taking the $q^n$-coefficient the claim follows since only terms with $i \leq n$ contribute.
Part (b) in case $h=1$ reduces to part (a) by factoring out $f_0 = 1 + \sum_{i \geq 1} a_{i0} p^i$
and applying (a) twice. The case of general $h$ reduces to case $h=1$ by multiplying out.
\end{proof}

The second, upcoming step is to compare the $[ - ]_{p^0}$ and $[ - ]_{z^{-1}}$ coefficient.
Together with Prop~\ref{lemma:p0 coefficient} it immediately implies property (b) of Theorem~\ref{thm:Characterization2}.

\begin{prop}
For any $k \geq 0$ we have
\[
\mathrm{Res}_{z=0} (\wp - 4 G_2)^{k+\frac{1}{2}} = 2 \left[ \sqrt{ \wp(p,q) - 4G_2(q) }^{2k+1} \right]_{p^0} + O(q^{k+1})
\]
\end{prop}
\begin{proof}
We let $z = 2\pi i x$ and define
\[ f(x) = \wp(2 \pi i x) - 4 G_2, \quad g(x) := \sqrt{\wp(2 \pi i x) - 4 G_2}.  \]
(In our convention $\wp(z)$ is double-periodic under $z \mapsto z + m + n \cdot 2 \pi i \tau$ for $m,n \in \BZ$. The scaled function $\wp(2 \pi i x)$ is then double-periodic under $x \mapsto x+m + n \tau$ and is the usual convention for the Weierstra{\ss} elliptic function found in the literature.)
The function $g(x)$ is a priori multivalued.
However, in a neighbourhood of $x=0$ we can choose the unique branch such that the Laurent expansion of $g(x)$ starts with $1/(2 \pi i x)$.
Recall that the Weierstrass function $\wp(2 \pi i x)$ and hence also $f(x)$ has two zeros $x_1,x_2$ in each fundamental domain, counted with multiplicities.
(Since $f(-x) = f(x)$ is even, we have $x_2 = -x_1$ up to translation by an element of the lattice $\BZ \tau + \BZ$.)
When extending $g(x)$ to the whole plane $\BC_{x}$ we hence run into the sign ambiguity when extending beyond the zeros.
The possible solution is to extend the function away from an appropriate branch cut of these two zeros (and its translates).

In the limit $|q| \ll 1$ the situation is as follows: Observe that
\[ f(x)|_{q=0} = \frac{1}{4} \left( \frac{e^{2 \pi i x} + 1}{ e^{2 \pi i x} - 1} \right)^2 \]
has a double zero at $x=\frac{1}{2}$ and its $\BZ$-translates,
so for small $q$ the two zeros of $f(x)$ are close to $x=1/2$ (modulo $\BZ$-translates)
and so we can choose the branch cut between these two adjacent zeros.
We now integrate over the boundary $\gamma$ of a fundamental domain with left corner at $a_0$,
for some $|q| \ll 1$, as in Figure~\ref{FigureBranchcut}.

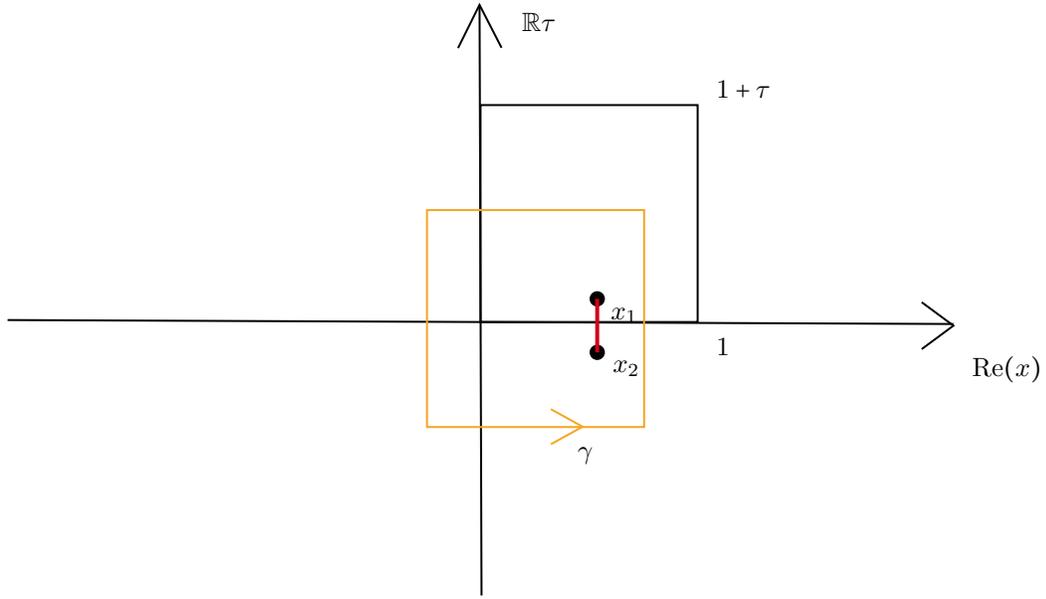
\begin{figure}
\centering


\tikzset{every picture/.style={line width=0.75pt}} 

\begin{tikzpicture}[x=0.75pt,y=0.75pt,yscale=-1,xscale=1]

\draw    (321,1.32) -- (322,299.32) ;
\draw    (83,160.32) -- (560,162.32) ;
\draw   (544,151.32) -- (560,163.16) -- (544,175) ;
\draw   (310.16,23.05) -- (320.98,1.24) -- (332.11,22.89) ;
\draw   (321.5,51.82) -- (431,51.82) -- (431,161.32) -- (321.5,161.32) -- cycle ;
\draw  [fill={rgb, 255:red, 0; green, 0; blue, 0 }  ,fill opacity=1 ] (376.99,176.68) .. controls (376.96,174.82) and (378.45,173.29) .. (380.31,173.27) .. controls (382.16,173.24) and (383.69,174.73) .. (383.72,176.58) .. controls (383.74,178.44) and (382.26,179.97) .. (380.4,179.99) .. controls (378.54,180.02) and (377.02,178.53) .. (376.99,176.68) -- cycle ;
\draw  [fill={rgb, 255:red, 0; green, 0; blue, 0 }  ,fill opacity=1 ] (376.99,149.68) .. controls (376.96,147.82) and (378.45,146.29) .. (380.31,146.27) .. controls (382.16,146.24) and (383.69,147.73) .. (383.72,149.58) .. controls (383.74,151.44) and (382.26,152.97) .. (380.4,152.99) .. controls (378.54,153.02) and (377.02,151.53) .. (376.99,149.68) -- cycle ;
\draw  [color={rgb, 255:red, 245; green, 166; blue, 35 }  ,draw opacity=1 ] (294.5,104.82) -- (404,104.82) -- (404,214.32) -- (294.5,214.32) -- cycle ;
\draw  [color={rgb, 255:red, 245; green, 166; blue, 35 }  ,draw opacity=1 ] (357,205.32) -- (373,214.16) -- (357,223) ;
\draw [color={rgb, 255:red, 208; green, 2; blue, 27 }  ,draw opacity=1 ][line width=1.5]    (380.35,149.63) -- (380.35,176.63) ;

\draw (439,38) node [anchor=north west][inner sep=0.75pt]   [align=left] {$1+\tau$};
\draw (385.72,152.58) node [anchor=north west][inner sep=0.75pt]   [align=left] {$x_1$};
\draw (369,223) node [anchor=north west][inner sep=0.75pt]   [align=left] {$\gamma$};
\draw (568,177) node [anchor=north west][inner sep=0.75pt]   [align=left] {$\mathrm{Re}(x)$};
\draw (341,4) node [anchor=north west][inner sep=0.75pt]   [align=left] {$\BR \tau$};
\draw (386.72,179.58) node [anchor=north west][inner sep=0.75pt]   [align=left] {$x_2$};
\draw (439,168) node [anchor=north west][inner sep=0.75pt]   [align=left] {$1$};

\end{tikzpicture}
\caption{We integrate over the orange path $\gamma$. We assume $|q| \ll 1$, so that the roots $x_1, x_2$ of $f(x)$ are close to $x=1/2$.
 The function $g(x) = \sqrt{f(x)}$ is defined away from the branch cut (the red line).
\label{FigureBranchcut}}
\end{figure}

Consider the winding number\footnote{The winding number along a closed path of $f$ along $\gamma$ is defined by $\int_{f \circ \gamma} \frac{dz}{z} = \int_{\gamma} \frac{f'(x)}{f(x)} dx$. It is the number of times that $f(\gamma(t))$ goes in counterclockwise direction around the origin.}
 $w(a) \in \BZ$ of $f(x)$
along the straight path from $a \in \BR$ to $a + \tau$.
We claim that $w(a) = -1$ for $0 < a \ll 1$, and $w(a)=1$ for some $a<1$ sufficiently close to $1$.
(To see this, note that $w(a)$ is locally constant in $a$ and only jumps when the vertical line crosses either a zero or a pole of $f$, where it jumps by $1$ when we cross a zero from left to right, and by $-1$ when we cross a pole. Since $f$ is $1$-periodic, $w(a)$ is $1$-periodic in $a$,
and since $f$ is even, one has that $w(-a) = -w(a)$. Hence for $\epsilon$ sufficiently small we obtain $w(1-\epsilon) = w(-\epsilon) = w(\epsilon) + 2$ and $w(1-\epsilon) = -w(\epsilon-1) = -w(\epsilon)$, and hence $w(\epsilon) + 2 = - w(\epsilon)$, so $w(\epsilon)=-1$.)
Further, let
$w'(b)$ be the winding number of $f$ when moving from $ b\tau$ to $b \tau + 1$ along a straight horizontal line.
For $q=0$ one easily sees that $w'(b) = 0$ for $b>>0$.
Hence, for $|q| \ll 1$, we have that $w'(b) = 0$ for $b$ close to $1/2$.

We conclude that for $|q| \ll 1$, the function $g(x)$ (as defined away from its branch cut)
satisfies $g(x+1) = g(x)$ but $g(x+\tau) = -g(x)$.
Hence we find the evaluation
\begin{equation} \label{dadf}
\int_{\gamma} g(x)^{2k+1} dx = 2 \int_{a_0}^{a_0 + 1} g(x)^{2k+1} dx = 2 \left[ \sqrt{ \wp(p,q) - 4G_2(q) }^{2k+1} \right]_{p^0}.
\end{equation}

On the other hand, we may apply the Cauchy integral formula to the integral $\int_{\gamma} g(x)^{2k+1} dx$.
The contribution from the pole at $x=0$ is the residue
\[ \frac{1}{2 \pi i} \mathrm{Res}_{x=0} g(x)^{2k+1} = \mathrm{Res}_{z=0} (\wp - 4 G_2)^{k+\frac{1}{2}}. \]
The contribution from the branch cut is given by
\[ I = \pm 2 \int_{x_1}^{x_2} g(x)^{2k+1} dx \]
where we choose one of the two branches of $g(x)$ in a neighbourhood of the line between the zeros $x_1, x_2$.
We will show that for $|q| \ll 1$ we have $I = O(q^{k+1})$.

Consider the asymptotic expansion of $f(x)$ near $x=1/2$ in $q$:
\[ f(x) = \frac{1}{4} \frac{(p+1)^2}{(p-1)^2} + (p - 6 + p^{-1}) q + \ldots \]
where $p=e^{2 \pi i x}$. 
We find the asymptotic expansion of the zeros $x_1, x_2$ to be
\[ x_i = \frac{1}{2} \pm \frac{4 i \sqrt{2}}{\pi} q^{1/2} + \ldots \]
where the higher order terms are multiples of $q^{3/2}, q^{5/2}$ etc.
(Note that since $f(-x) = f(x)$ and $f(x+1) = f(x)$, we have $f(1/2 - x) = f(1/2+x)$.)
The function $g(x)$ has the asymptotic expansion \eqref{expansion_g}.
Hence if we set $x=\frac{1}{2} + 2 \pi i T$ so that $p=- e^{T}$ we obtain the expansion
\[ g\left(\frac{1}{2} + T \right) = T h_0(T) + \frac{h_1(T)}{T} q + \frac{ h_2(T) }{T^3} q^2 + \ldots \]
where $h(T)$ are power series in $T$.
Hence
\[ g\left(\frac{1}{2} + T \right)^{2k+1} = T^{2k+1} \tilde{h}_0(T) + T^{2k-1} \tilde{h}_1(T) q + T^{2k-3} \tilde{h}_2(T) q^2 + \ldots \]
for some power series $\tilde{h}_i(T)$.
We obtain that
\[
\pm 2 \int_{x_1}^{x_2} g(x)^{2k+1} dx = \int_{-\frac{4i \sqrt{2}}{\pi} \sqrt{q}}^{\frac{4i \sqrt{2}}{\pi} \sqrt{q}} 
\left( T^{2k+1} \tilde{h}_0(T) + T^{2k-1} \tilde{h}_1(T) q + T^{2k-3} \tilde{h}_2(T) q^2 + \ldots \right)  dT + \ldots
= O(q^{k+1}),
\]
where $\ldots$ stands for terms of higher order in $q$.
Combining \eqref{dadf} with the Cauchy integral formula gives
\[
\int_{\gamma} g(x)^{2k+1} dx = \mathrm{Res}_{z=0} (\wp - 4 G_2)^{k+\frac{1}{2}} + \pm 2 \int_{x_1}^{x_2} g(x)^{2k+1} dx \]
and hence completes the claim.
\end{proof}

\subsection{Characterization of the $B$ and $C$ series}
The characterization of $B_k(q)$ is similar:
\begin{prop} \label{prop:Char B}
The series $B_k(q)$, $k\geq 0$ defined by \eqref{B} is the unique series of power series which satisfies the following conditions:
\begin{enumerate}[itemsep=0pt]
\item[(a)] $B_k(q)$ is a quasi-modular form of weight $2k+4$ satisfying $\frac{d}{dG_2} B_k = 2 B_{k-1} - 2 A_{k+1}$.
\item[(b)] For every $n \geq 0$ there exists a polynomial $q_n(k)$ of degree $2n-2$ such that
for every $k \geq n-1$ we have
\[ \Big[ B_k(q) \Big]_{q^n} = \frac{k!}{(2k+1)! (-2)^k} q_n(k). \]
\item[(c)] $B_0 = q + O(q^2)$.
\end{enumerate}
\end{prop}

\begin{proof}
The proof follows along the lines of the proof of Theorem~\ref{thm:Characterization2}. The uniqueness is completely parallel. To check that $B_k$ satisfies (a) and (c) is again straightforward. To show (b) one uses the same residue argument to show that
\[
\mathrm{Res}_{z=0} \left( (\wp - 4 G_2)^{k+\frac{3}{2}} (\wp + 2 G_2)\right) = 2 \left[ \sqrt{ \wp(p,q) - 4G_2(q) }^{2k+3} (\wp(p,q)+2 G_2(q)) \right]_{p^0} + O(q^{k+2})
\]
where $O(q^{k+2})$ appears instead of $O(q^{k+1})$ here because we integrate $\wp -4 G_2$ with exponent $k+3/2$ instead of $k+1/2$.
It remains to show the desired polynomiality of the first term on the right. Since
$\wp(p,q)+2 G_2(q)$ has constant term zero,
the application of Lemma~\ref{lemma:combinatorics}(b) shows that there exists polynomials $p_n(k)$ of degree $2n-1$ such that for all $k$:
\[ \left[ (\wp(p,q) - 4G_2(q) )^{k+\frac{3}{2}} (\wp(p,q) + 2 G_2(q)) \right]_{p^0} = \frac{1}{2^{2k}}\sum_{n \geq 0} p_n(k) q^n. \]
Moreover, since we have $[\wp(p,q)+2 G_2]_{p^0}=0$ this expression vanishes whenever $k=-3/2$, so the polynomial $p_n(k)$ is divisible by $2k+3$. The claim follows.
\end{proof}

\begin{proof}[Proof of Theorem~\ref{thm:Point insertions}]
By Lemma~\ref{lemma:splitting formula} we know that
\[
\left\langle \prod_{i=1}^{n} \tau_{k_i}(\pt) \right\rangle^{\GW}
=
\frac{\widetilde{B}_{k_1}(q) \cdots \widetilde{B}_{k_n}(q)}{\Delta(q)},
\]
for some power series $\widetilde{B}_k(q)$.
If we know (i) then it follows that $\widetilde{B}_k(q)$ satisfies property (b) of Proposition~\ref{prop:Char B}.
Propoerty (a) follows immediately from \eqref{thm:HAE}.
Property (c) is clear. Hence one gets $\widetilde{B}_k = B_k$.
The converse direction is also by Proposition~\ref{prop:Char B}.
\end{proof}

We also record the properties of $C_{k \ell}$.
Since by definition we have $C_{k0} = C_{0k}= B_{k-1}$ we will assume $k, \ell \geq 1$.
\begin{prop} \label{prop:Char C}
The series $C_{k \ell}(q)$, $k, \ell\geq 1$ defined by \eqref{C} satisfies the following conditions:
\begin{enumerate}[itemsep=0pt]
\item[(a)] $C_{k \ell}$ is a quasi-modular form of weight $2k+2 \ell +2$ satisfying
\[
\frac{d}{dG_2} C_{k \ell} = 2 C_{k-1,\ell} + 2 C_{k, \ell-1} - 2 A_k A_\ell.
\]
\item[(b)] For every $n \geq 0$ there exists a polynomial $P_n(k,\ell)$ of degree $2n-2$ such that for all $k,\ell \geq n-1$:
\[ [ C_{k,\ell} ]_{q^n} = \frac{1}{(-4)^{k} (2k-1)!! (-4)^{\ell} (2 \ell-1)!!} P_n(k,\ell). \]
\item[(c)] For every $\ell \geq 1$ and every $n \geq 0$ there exists a polynomial $p_n(k)$ of degree $2n-2$ such that for all $k \geq n-1$:
\[ [ C_{k,\ell} ]_{q^n} = \frac{1}{(-4)^{k} (2k-1)!!}p_n(k). \]
\item[(d)] $C_{k \ell} = O(q)$. 
\item[(e)] $C_{k \ell} = C_{\ell k}$ for all $k,\ell$.
\end{enumerate}
The $C_{k\ell}$ are uniquely determined by properties (a,c,e).
\end{prop}
\begin{proof}
The uniqueness is straightforward (determine first $C_{k1}$ for all $k$, then $C_{k2}$ for all $k$, and so on).
Conversely, (a,e) is clear and (d) follows from (b) using $n=0$.
To show (b) one argues as in the proof of Proposition~\ref{prop:Char B} (in particular, using that for fixed $z_2$, the $p^0$-coefficient of $\wp(z_1-z_2) + 2 G_2$ vanishes).
For part (c) one argues similarly using Lemma~\ref{lemma:combinatorics2} below, which is an analogue of Lemma~\ref{lemma:combinatorics} and whose proof is left to the reader.
\end{proof}
\begin{lemma} \label{lemma:combinatorics2}
\begin{enumerate}
\item[(a)] Let $f,g \in 1+qR[[q]]$ power series with constant term one and coefficients in a ring $R$.
Then $[f^k g^l]_{q^n}$ is a polynomial in $k,l$ of degree $\leq n$.
\item[(b)] Let $f(p_1,q) \in \BC[[p_1,q]]$ and $g(p_2,q) \in \BC[[p_2,q]]$ be power series, both with constant term one, and let $h(p_1,p_2,q) \in \CC[[p_1,p_2,q]]$ be any power series. Then the coefficient $[f^k g^l h]_{p_1^{m_1} p_2^{m_2} q^n}$ is a polynomial in $k,\ell$ of degree $\leq m_1+m_2+n$, and of degree $\leq m_1+m_2+n-1$ if $h$ has zero constant term.
\end{enumerate}
\end{lemma}

\begin{rmk}
Using the above characterizations of the functions $A,B,C$
one can show that
Conjecture~\ref{conj:GW stationary}
implies Conjecture~\ref{conj:polynomiality}
in the stationary case.
\end{rmk}

\subsection{An upgrade of Conjecture~\ref{conj:polynomiality}}
\label{subsec:upgraded polynomiality}
We record here the following conjectural strengthening of the polynomiality property:

\begin{conj}[Upgrade of Conjecture~\ref{conj:polynomiality}]
For any subsets $I_x \subset \{ 1, \ldots, x \}$ for $x \in (r, s ,t ,u )$,
and for $k_i, \ell_i, m_i, n_i$ fixed whenever $i$ does not lie in $I_r, I_s, I_t, I_u$ respectively (and satisfying $k_i, m_i \geq 1$),
there exists a polynomial $p$ of degree $\beta^2+2-2u-t+r$ such that
for all $k_i, \ell_i, m_i, n_i$ satisfying \eqref{polynomial range} for $i$ in $I_r,I_s,I_t, I_u$ we have
\[
\left\llangle 
\prod_{i=1}^{r} \tau_{k_i}(1) 
\prod_{i=1}^{s} \tau_{\ell_i}(\beta)
\prod_{i=1}^{t} \tau_{m_i}(\delta_i)
\prod_{i=1}^{u} \tau_{n_i}(\pt)
\right\rrangle^S_{\beta}
=
p\left( (k_{i} )_{i \in I_r}, ( \ell_i )_{i \in I_{s}}, (m_i)_{i \in I_t}, (n_i)_{i \in I_u} \right).
\]
%
\end{conj}

\begin{example} For $k,\ell \geq \beta^2/2 + 3$ we have:
\begin{align*} \llangle \tau_{k}(1) \tau_{\ell}(1) \rrangle^S_{\beta^2/2=-1}
& = 2 {\left(k + l - 3\right)} {\left(2  k + 2  l - 5\right)} \\
\llangle \tau_{k}(1) \tau_{\ell}(1) \rrangle^S_{\beta^2/2=0}
& = 16 {\left(k + l - 3\right)} \big(4  k^{3} + 4  k^{2} l + 4  k l^{2} + 4  l^{3} \\
& \quad - 16  k^{2} - 12  k l - 16  l^{2} + 29  k + 29  l - 29\big) \\
\llangle \tau_{k}(1) \tau_{\ell}(1) \rrangle^S_{\beta^2/2=1}
& =
8 {\left(k + l - 3\right)} \big(32  k^{5} + 32  k^{4} l + 128  k^{3} l^{2} + 128  k^{2} l^{3} + 32  k l^{4} + 32  l^{5} \\
& - 336  k^{4} - 448  k^{3} l - 704  k^{2} l^{2} - 448  k l^{3} - 336  l^{4} + 1392  k^{3} + 1328  k^{2} l \\
& + 1328  k l^{2} + 1392  l^{3} - 2236  k^{2} - 1624  k l - 2236  l^{2} + 1780  k + 1780  l - 1049\big)
\end{align*}
On the other hand, for $\ell < 3$ we have polynomiality only in $k$ for $k \geq \beta^2/2+3$.
For example,
\begin{align*}
\llangle \tau_{k}(1) \tau_{1}(1) \rrangle^S_{\beta^2/2=0}
& = 32 \, {\left(k^{2} - 2 \, k + 3\right)} {\left(2 \, k - 3\right)} {\left(k - 1\right)} 
&& = P(k,1) + 32 \, k^{3} - 112 \, k^{2} + 192 \, k - 96 \\
\llangle \tau_{k}(1) \tau_{2}(1) \rrangle^S_{\beta^2/2=0}
& = 16 \, {\left(2 \, k^{3} - 5 \, k^{2} + 12 \, k - 6\right)} {\left(2 \, k - 1\right)}
&& = P(k,2) + 48 \\
\llangle \tau_{k}(1) \tau_{3}(1) \rrangle^S_{\beta^2/2=0}
& = 16 \, {\left(4 \, k^{3} - 4 \, k^{2} + 29 \, k + 22\right)} k 
&& = P(k,3)
\end{align*}
where
\[ P(k,l) = 16 {\left(k + l - 3\right)} \big(4  k^{3} + 4  k^{2} l + 4  k l^{2} + 4  l^{3} 
- 16  k^{2} - 12  k l - 16  l^{2} + 29  k + 29  l - 29\big) \]
is the polynomial answer for $k, \ell \geq 3$.
By definition we have excluded here the case $\ell=0$.
%
\end{example}

\appendix

\section{Further examples}
We list some more computations for Gromov--Witten invariants
in the polynomial range.
This assumes Conjecture~\ref{conj:GW stationary} (in order that we can apply our algorithm of Theorem~\ref{thm:recursion}),
and Conjecture~\ref{conj:polynomiality} to get a bound on the degree of the polynomial.
However, the computations are also always a check on the polynomiality since we computed more terms than was required to fix the degree of the polynomial.

\begin{example} For $k,\ell \geq \beta^2/2 - 1$ we have:
\begin{align*} 
\llangle \tau_{k}(\pt) \tau_{\ell}(\pt) \rrangle^S_{\beta^2/2=1} & = 1 \\
\llangle \tau_{k}(\pt) \tau_{\ell}(\pt) \rrangle^S_{\beta^2/2=2} & = 
8  k^{2} + 8  l^{2} - 12  k - 12  l + 20 \\
\llangle \tau_{k}(\pt) \tau_{\ell}(\pt) \rrangle^S_{\beta^2/2=3} 
& = \frac{64}{3}  k^{4} + 64  k^{2} l^{2} + \frac{64}{3}  l^{4} - \frac{512}{3}  k^{3} - 96  k^{2} l - 96  k l^{2} - \frac{512}{3}  l^{3} + \frac{1712}{3}  k^{2} \\
& \quad + 144  k l + \frac{1712}{3}  l^{2} - \frac{1024}{3}  k - \frac{1024}{3}  l - 64
\end{align*}
\end{example}

\begin{example} 
Let $\alpha_1, \alpha_2 \perp \beta$ with $\alpha_1^2 = \alpha_2^2=0$ and $\alpha_1 \cdot \alpha_2=1$. For $k,\ell \geq \max(\beta^2/2,1)$ we have:
\begin{align*}
\llangle \tau_{k}(\alpha_1) \tau_{\ell}(\alpha_2) \rrangle^S_{\beta^2/2=0} & = -1/4 \\
\llangle \tau_{k}(\alpha_1) \tau_{\ell}(\alpha_2) \rrangle^S_{\beta^2/2=1} & = 
-2 \, k^{2} - 2 \, k l - 2 \, l^{2} + 7 \, k + 7 \, l - \frac{29}{2} \\
\llangle \tau_{k}(\alpha_1) \tau_{\ell}(\alpha_2) \rrangle^S_{\beta^2/2=2} & =  
-\frac{16}{3} \, k^{4} - \frac{32}{3} \, k^{3} l - \frac{64}{3} \, k^{2} l^{2} - \frac{32}{3} \, k l^{3} - \frac{16}{3} \, l^{4} + 64 \, k^{3} + \frac{368}{3} \, k^{2} l \\
& + \frac{368}{3} \, k l^{2} + 64 \, l^{3} - \frac{1016}{3} \, k^{2} - 432 \, k l - \frac{1016}{3} \, l^{2} + 678 \, k + 678 \, l - 606
\end{align*}
\end{example}

\begin{example} For $k, \ell \geq \beta^2/2$ we have:
\begin{align*}
\llangle \tau_{k}(\pt) \tau_{\ell}(F) \rrangle^S_{\beta^2/2=0} & = 1 \\
\llangle \tau_{k}(\pt) \tau_{\ell}(F) \rrangle^S_{\beta^2/2=1} & = 8  k^{2} + 16  l^{2} - 12  k - 24  l + 6 \\
\llangle \tau_{k}(\pt) \tau_{\ell}(F) \rrangle^S_{\beta^2/2=2} & =
\frac{64}{3}  k^{4} + 128  k^{2} l^{2} + 64  l^{4} - \frac{512}{3}  k^{3} - 192  k^{2} l - 192  k l^{2} - 512  l^{3} \\
& \quad + \frac{1376}{3}  k^{2} + 288  k l + 1344  l^{2} - \frac{520}{3}  k - 472  l - 512
\end{align*}
\end{example}

\begin{example}
For $k_1, k_2, k_3 \geq \beta^2/2+3$ we have
\begin{align*}
\llangle \tau_{k_1}(1) \tau_{k_2}(1) \tau_{k_3}(1) \rrangle^S_{\beta^2/2=-1}
& = 4  {\left(k_{1} + k_{2} + k_{3} - 4\right)} {\left(2  k_{1} + 2  k_{2} + 2  k_{3} - 7\right)} {\left(k_{1} + k_{2} + k_{3} - 3\right)} \\
\llangle \tau_{k_1}(1) \tau_{k_2}(1) \tau_{k_3}(1) \rrangle^S_{\beta^2/2=0}
& = 
32  \left(k_{1} + k_{2} + k_{3} - 4\right) \left(2  k_{1} + 2  k_{2} + 2  k_{3} - 7\right) (2  k_{1}^{3} + 2  k_{1}^{2} k_{2} + 2  k_{1} k_{2}^{2} \\
& + 2  k_{2}^{3} + 2  k_{1}^{2} k_{3} + 2  k_{2}^{2} k_{3} + 2  k_{1} k_{3}^{2} + 2  k_{2} k_{3}^{2} + 2  k_{3}^{3} - 9  k_{1}^{2} - 6  k_{1} k_{2} \\
& - 9  k_{2}^{2} - 6  k_{1} k_{3} - 6  k_{2} k_{3} - 9  k_{3}^{2} + 17  k_{1} + 17  k_{2} + 17  k_{3} - 21)  \\
\llangle \tau_{k_1}(1) \tau_{k_2}(1) \tau_{k_3}(1) \rrangle^S_{\beta^2/2=1}
& = 16  {\left(k_{1} + k_{2} + k_{3} - 4\right)} {\left(2  k_{1} + 2  k_{2} + 2  k_{3} - 7\right)} (16  k_{1}^{5} + 16  k_{1}^{4} k_{2} + 64  k_{1}^{3} k_{2}^{2} \\
& + 64  k_{1}^{2} k_{2}^{3} + 16  k_{1} k_{2}^{4} + 16  k_{2}^{5} + 16  k_{1}^{4} k_{3} + 64  k_{1}^{2} k_{2}^{2} k_{3} + 16  k_{2}^{4} k_{3} + 64  k_{1}^{3} k_{3}^{2} \\
& + 64  k_{1}^{2} k_{2} k_{3}^{2} + 64  k_{1} k_{2}^{2} k_{3}^{2} + 64  k_{2}^{3} k_{3}^{2} + 64  k_{1}^{2} k_{3}^{3} + 64  k_{2}^{2} k_{3}^{3} + 16  k_{1} k_{3}^{4} + 16  k_{2} k_{3}^{4} \\
& + 16  k_{3}^{5} - 176  k_{1}^{4} - 224  k_{1}^{3} k_{2} - 384  k_{1}^{2} k_{2}^{2} - 224  k_{1} k_{2}^{3} - 176  k_{2}^{4} - 224  k_{1}^{3} k_{3} \\
& - 192  k_{1}^{2} k_{2} k_{3} - 192  k_{1} k_{2}^{2} k_{3} - 224  k_{2}^{3} k_{3} - 384  k_{1}^{2} k_{3}^{2} - 192  k_{1} k_{2} k_{3}^{2} - 384  k_{2}^{2} k_{3}^{2}\\
&  - 224  k_{1} k_{3}^{3} - 224  k_{2} k_{3}^{3} - 176  k_{3}^{4} + 792  k_{1}^{3} + 744  k_{1}^{2} k_{2} + 744  k_{1} k_{2}^{2} + 792  k_{2}^{3} \\
& + 744  k_{1}^{2} k_{3} + 432  k_{1} k_{2} k_{3} + 744  k_{2}^{2} k_{3} + 744  k_{1} k_{3}^{2} + 744  k_{2} k_{3}^{2} + 792  k_{3}^{3} - 1402  k_{1}^{2}\\
&  - 980  k_{1} k_{2} - 1402  k_{2}^{2} - 980  k_{1} k_{3} - 980  k_{2} k_{3} - 1402  k_{3}^{2} + 1221  k_{1} \\
& + 1221  k_{2} + 1221  k_{3} - 873)
\end{align*}
\end{example}

Mathematisches Institut, Universit\"at Heidelberg

georgo@uni-heidelberg.de

    \end{document}